\documentclass[english,letterpaper,11pt,reqno]{amsart}

\usepackage{appendix}
\usepackage{amsbsy}
\usepackage{amsfonts}
\usepackage{amsmath}
\usepackage{amssymb}
\usepackage{amsthm}
\usepackage{graphicx}
\usepackage{ifthen}
\usepackage{textcomp}
\usepackage{enumitem, color, amssymb}
\usepackage{dsfont}
\usepackage{mathrsfs}
\usepackage{calrsfs}
\usepackage[bookmarksnumbered,colorlinks]{hyperref}
\hypersetup{colorlinks=true, linkcolor=blue, citecolor=blue}
\usepackage{soul}
\usepackage{cancel}
\usepackage{scalerel}
\usepackage{comment}

\newtheorem{thm}{Theorem}
\newtheorem{lemma}{Lemma}
\newtheorem{cor}{Corollary}

\newtheorem{defn}{Definition}
\newtheorem{prop}{Proposition}

\newtheorem*{definition-non}{Definition}
\newtheorem*{theorem-non}{Theorem}
\newtheorem*{proposition-non}{Proposition}
\newtheorem*{lemma-non}{Lemma}
\newtheorem*{corollary-non}{Corollary}
\newtheorem*{conjecture-non}{Conjecture}

\newcommand{\beqa}{\begin{eqnarray}}
\newcommand{\beq}{\begin{equation}}
\newcommand{\eeqa}{\end{eqnarray}}
\newcommand{\eeq}{\end{equation}}

\newcommand{\lra}{\longrightarrow}

\newcommand{\RR}{\mathbb{R}}

\newcommand{\vep}{\varepsilon}

\newcommand\cd[2]{\nabla_{\!#1}{#2}}
\newcommand\cdf[2]{\nabla^{\scriptscriptstyle f}_{\!#1}{#2}}

\newcommand\imp{\hspace{.2in}\Rightarrow\hspace{.2in}}

\newcommand\comma{\hspace{.2in},\hspace{.2in}}
\newcommand\commas{\hspace{.1in},\hspace{.1in}}
\newcommand\commass{\hspace{.05in},\hspace{.05in}}

\newcommand{\gR}{g_{\scalebox{0.4}{$R$}}}
\newcommand\gS{g_{\scalebox{0.4}{$S$}}}

\newcommand{\Jf}{J_{\scaleto{f}{4.5pt}}}
\newcommand{\If}{I_{\scaleto{f}{4.5pt}}}

\newcommand{\gammaf}{\gamma_{\scaleto{f}{4.5pt}}}

\makeatletter
\newcommand*{\defeq}{\mathrel{\rlap{%
                     \raisebox{0.24ex}{$\m@th\cdot$}}%
                     \raisebox{-0.24ex}{$\m@th\cdot$}}%
                     =}
\makeatother

\begin{document}
\title[]{The Eisenhart Lift and Hamiltonian Systems}
\author[]{Amir Babak Aazami}
\address{Clark University, Worcester, MA, USA}
\email{Aaazami@clarku.edu}

\maketitle
\begin{abstract}
It is well known in general relativity that trajectories of Hamiltonian systems lift to geodesics of pp-wave spacetimes, an example of a more general phenomenon known as the ``Eisenhart lift."  We review and expand upon the benefits of this correspondence for dynamical systems theory.  One benefit is the use of curvature and conjugate points to study the stability of Hamiltonian systems.  Another benefit is that this lift unfolds a Hamiltonian system into a family of ODEs akin to a moduli space. One such family arises from the conformal invariance of lightlike geodesics, by which any Hamiltonian system unfolds into a ``conformal class" of non-diffeomorphic ODEs with solutions in common.  By utilizing higher-index versions of pp-waves, a similar lift and conformal class are shown to exist for certain second-order complex ODEs.  Another such family occurs by lifting to a Riemannian metric that is dual to a pp-wave, a process that in certain cases yields a ``square root" for the Hamiltonian.  We prove a two-point boundary result for the family of ODEs arising from this lift, as well as the existence of a constant of the motion generalizing conservation of energy.
\end{abstract}

\section{Introduction}
Consider the Hamiltonian system
\beqa
\label{eqn:0}
\ddot{x} = -\nabla_{\!x}V,
\eeqa
where the potential $V(t,x)$ is a $C^2$ function of $t \in \RR$ and $x \in \RR^n$, and where $\nabla_{\!x}V$ is the gradient with respect to $x$.  Almost a century ago, L.~Eisenhart \cite{eisenhart} showed that \eqref{eqn:0} directly relates to the geodesics of what are now called \emph{pp-wave} spacetimes, a class of Lorentzian manifolds first discovered mathematically by H.~Brinkmann \cite{brinkmann} and later shown to model gravitational waves in general relativity; see, e.g., \cite{beem,AMS}.  In particular, solutions of an $n$-dimensional Hamiltonian system \eqref{eqn:0} lift to geodesics of an $(n+2)$-dimensional pp-wave, a relationship that stands apart from the usual symplectic formulation of \eqref{eqn:0} via the cotangent bundle. Since Eisenhart's work, other relationships have been found between classical mechanics and the geodesics of metrics; see, e.g., \cite{gordon2,ong,casetti2,bartolo,pettini} for approaches involving Riemannian, Lorentzian, and Finslerian manifolds.  The ``Eisenhart lift" (or ``Eisenhart-Duval lift") is now very well known in general relativity and high energy physics (see \cite{cariglia} for an overview), and has seen various applications and generalizations; see, e.g.,  \cite{barg,duval_gibbons,candela,flores2,minguzzi2,minguzzi3}, some of which we review below.
The key ingredient is the presence in pp-waves of a \emph{parallel lightlike} vector field; this is what enables the connection with \eqref{eqn:0}.
\vskip 6pt
In this paper we engage with the Eisenhart lift\,---\,not in the service of geometry or physics, but rather to promote its use in the theory of dynamical systems.  Of course, the interplay between dynamical systems and Lorentzian geometry is already rich: Besides the references listed above, see also \cite{sullivan,fathi,BernSuhr,monclair} for the analogy between Lyapunov functions and time functions, and \cite{benenti,gibbons_s,cariglia2} on using conserved quantities of a dynamical system to identify spacetimes with higher-rank Killing tensors.  Our intention here is to make a small contribution to this vast literature, focusing, like \cite{casetti2}, on the ways in which the geometry of pp-waves can be used to gain insight into \eqref{eqn:0} itself.  The Lorentzian geometry we call upon is completely classical\,---\,e.g., the conformal properties of lightlike geodesics and their conjugate points (see \cite{beem,minguzzi}), or how Ricci curvature yields the focusing of geodesics \`a la the Penrose-Hawking singularity theorems \cite{beem,o1983}\,---\,but as we show below, we also go beyond Lorentzian geometry.  In particular, we are motivated by, and provide partial answers to, the following two questions:
\vskip 5pt
i.~\emph{Given a solution of \eqref{eqn:0}, to what other metrizable but non-diffeomoprhic ODEs will it lift as a solution?   What about variants of \eqref{eqn:0}?}
\vskip 4pt
By lifting solutions $x(t)$ of \eqref{eqn:0} to \emph{lightlike} geodesics of a pp-wave, one can exploit the conformal invariance of the latter.  This fact is well known and has seen applications and generalizations; see, e.g., \cite{duval_gibbons, minguzzi2}.  We provide a brief review of this in Section \ref{sec:pp}, our goal being to highlight an important consequence for \eqref{eqn:0}: Lightlike conformal invariance uncovers a non-diffeomorphic ``conformal class" of ODEs to which $x(t)$ is, after a reparametrization, also a solution, one, in fact, with the same accumulation points, in the sense of ii.~below.  This is known, and has been utilized in settings more general than ours; see, e.g., \cite[Theorem~3.2,~eqn.~(34)ff.]{minguzzi2}.  To close our review, and to show what the conformal family of ODEs for \eqref{eqn:0} looks like (as we cannot find them in the literature), we present them explicitly in Theorem \ref{thm:1d}.  We then proceed in two new directions.  First, we show in Section \ref{sec:comp} that the Eisenhart lift and its ``conformal class" of ODEs also extend to the \emph{complex} plane (Theorem \ref{thm:main2}).  In particular, solutions of $\ddot{z} = F(z(t))$, where $F$ is holomorphic and $z(t) = x(t)+iy(t)$, lift directly to lightlike geodesics of split-signature pp-waves (i.e., those of signature $(\!--++)$); thus $\ddot{z} = F(z(t))$, too, has a conformal class of ODEs.  Second, we show in Section \ref{sec:Riem} that solutions of $\ddot{x} = -\nabla_{\!x}(V^2)$, with $V(x)$ time-independent, also lift to geodesics of a complete \emph{Riemannian} metric that is ``dual" to a pp-wave (Theorem \ref{prop:Kahler}), a lift that can be viewed as ``taking the square root" of $\ddot{x} = -\nabla_{\!x}(V^2)$.  For potentials $V(t,x)$ in general, this ``Riemannian Eisenhart lift" generalizes \eqref{eqn:0} to a family of ODEs for which we prove, in Theorem \ref{thm:Kahler}, a two-point boundary result as well as the existence of a constant of the motion generalizing conservation of energy:
\begin{theorem-non}[Generalized Hamiltonian system]
Let $V(t,x)$ be a $C^2$ function globally defined on $\RR\times \RR^n$.  For any two points $x_0,x_1 \in \RR^n$, there exist a constant $c$ and a $C^2$ function $\tau(t)$, both depending on $x_0,x_1$, such that the second-order ODE
\beqa
\label{eqn:tV0}
\ddot{x} = -2\Big(c+\!\!\int_0^t(\nabla_{\!x}\widetilde{V}\cdot \dot{x})\,dt\Big)\nabla_{\!x}\widetilde{V} \comma \widetilde{V}(t,x) \defeq V(\tau(t),x),
\eeqa
has a complete solution $x(t)$ passing through $x_0$ and $x_1$, which satisfies
\beqa
\label{eqn:CoE0}
\frac{1}{2}|\dot{x}(t)|^2 + \frac{1}{4}\dot{\tau}(t)^2 = \emph{\text{const.}}
\eeqa
If $V_t = 0$, then \eqref{eqn:tV0} reduces to the Hamiltonian system
$$
\ddot{x} = -\nabla_{\!x}(V+\bar{c})^2 \comma \bar{c}\defeq c-V(x_0),
$$
and \eqref{eqn:CoE0} the conservation of energy equation for the  potential $(V+\bar{c})^2$.
\end{theorem-non}
Thus the Eisenhart Lift, together with its complex and Riemannian variants, can be viewed as defining various ``moduli spaces of ODEs" for \eqref{eqn:0}.
\vskip 5pt
ii.~\emph{Regarding the stability of \eqref{eqn:0}, what conditions on the potential $V(t,x)$ will guarantee the accumulation of nearby trajectories?}
\vskip 4pt
In differential geometry, conjugate points along geodesics indicate when nearby geodesics ``accumulate," in a sense made precise by the concept of a \emph{Jacobi field} along a geodesic.  For pp-waves (and certain generalizations of them), these were studied in \cite{flores2} and some conditions for their presence/absence were given in \cite[Proposition~6.4, Remark~6.5]{flores2}, in terms of the convexity of the Hessian of $V$ in \eqref{eqn:0} and the sign of the pp-wave's sectional curvature.  In our Theorems \ref{thm:conj_1} and \ref{thm:last_conj}, we explore further conditions for their existence via the machinery of the classical Penrose-Hawking singularity theorems (see \cite{beem,o1983}); e.g., in Theorem \ref{thm:last_conj} we prove:
\begin{theorem-non}
Let $V(t,x)$ be a $C^2$ function and $x(t)$ a maximal, nonconstant solution of $\ddot{x} = -\nabla_{\!x}V$ satisfying the following properties:
\begin{enumerate}[leftmargin=.4in]
\item[i.] There is a time $t_0$ at which $\dot{x}^i(t_0) = 0$ for some $i=1,\dots,n$,
\item[ii.] $V_{ij}\big|_{(t_0,x(t_0))} \neq 0$ for some $j=1,\dots,n$,
\item[iii.] $\Delta_x V\big|_{(t,x(t))} \geq 0$.
\end{enumerate}
If $x(t)$ is defined for all time, then there is a one-parameter family of distinct solutions of $\ddot{x} = -\nabla_{\!x}V$ starting at some point $x(t_1)$ and accumulating at some later point $x(t_2)$.
\end{theorem-non}
\noindent The notion of ``accumulation point" specific to \eqref{eqn:0} is made precise in Definition \ref{def:conj2}. (See also \cite{casetti2} for further applications of Jacobi fields to \eqref{eqn:0}, particularly to Lyapunov exponents and chaos.)
\vskip 5pt
The outline of this paper is as follows: We provide a detailed review of the Eisenhart lift and the geometry of pp-waves in Section \ref{sec:pp}\,---\,note that everything presented therein is known.  Sections \ref{sec:conjugate}-\ref{sec:Riem} then apply the Eisenhart lift and its generalizations to show Theorems \ref{thm:conj_1}-\ref{thm:Kahler} above.

\section{Review of Hamiltonian systems, pp-waves, and the Eisenhart lift}
\label{sec:pp}
First, recall that a \emph{Lorentzian manifold} $(M,g)$ is a smooth manifold $M$ on which is defined a smooth, nondegenerate metric $g$ of signature $(-++\cdots+)$.  Because of the absence of positive-definiteness, any nonzero tangent vector $X \in TM$ comes in three flavors:
\beqa
\text{$X$ is}~\left\{
\begin{array}{rcl}
\text{``timelike"} & \text{if} & g(X,X) < 0,\\
\text{``spacelike"} & \text{if} & g(X,X) > 0,\\
\text{``lightlike"} & \text{if} & g(X,X) = 0.\nonumber
\end{array}\right.
\eeqa
(A tangent vector's designation as timelike, spacelike, or lightlike is known as its \emph{causal character}; in general, metrics of signature $(--\cdots-++\cdots+)$ are called \emph{semi-Riemannian}.)
The Lorentzian manifolds we have in mind are \emph{pp-waves}, which have a distinguished history.  They are the simplest examples of Lorentzian manifolds endowed with a parallel (i.e., covariantly constant) lightlike vector field, a class of metrics that was first introduced mathematically in \cite{brinkmann}.  From a physical perspective, pp-waves model radiation propagating at the velocity of light in general relativity; see, e.g., \cite[Chapter 13]{beem}, \cite{flores2}, and \cite{AMS}. The definition of pp-wave that we give here, taken from \cite{globke}, is a coordinate-independent version of the standard definition appearing in the physics literature:

\begin{defn}[pp-wave]
\label{def:pp00}
On a \emph{(}compact or noncompact\emph{)} manifold $M$, a Lorentzian metric $g$ is a \emph{pp-wave} if it admits a globally defined lightlike vector field $N$ that is parallel, $\nabla N = 0$ \emph{(}$\nabla$ is the Levi-Civita connection of $g$\emph{)}, and if its curvature endomorphism $R_{\scalebox{0.6}{$g$}}$ satisfies
\beqa
\label{def:ppwave}
R_{\scalebox{0.6}{$g$}}(X,Y)\,\cdot = 0~~~\text{for all $X,Y \in \Gamma(N^{\perp})$.}
\eeqa
\end{defn}

Locally, pp-waves always take the following form, now known as ``Brinkmann coordinates" (for our purposes, it suffices that $V$ be at least $C^2$):

\begin{prop}[Brinkmann coordinates]
\label{prop:Brink}
Let $(M,g)$ be an $(n+2)$-dimensional pp-wave with parallel lightlike vector field $N$.  Then there exist local coordinates $(v,t,x^1,\dots,x^n)$ in which $N = \partial_v$ and
\beqa
\label{eqn:metric}
g \defeq 2dvdt  -2V(t,x^1,\dots,x^n)(dt)^2+ \sum_{i=1}^n (dx^i)^2,
\eeqa
for some smooth function $V(t,x^1,\dots,x^n)$ independent of $v$.
If \eqref{eqn:metric} exists globally on $\RR^{n+2}$, then $(\RR^{n+2},g)$ is called a \emph{standard pp-wave}. 
\end{prop}

\begin{proof}
For a proof of this fact, consult \cite{globke}.  (We also mention here that, as proved in \cite{Leistner}, if a pp-wave is compact, then its universal cover is globally isometric to a standard pp-wave.)  
\end{proof}

The relationship between pp-waves and Hamiltonian systems is due to the geodesics of \eqref{eqn:metric}; in what follows, we will always regard our pp-waves as $(n+2)$-dimensional and our Hamiltonian systems as $n$-dimensional.

\begin{prop}[Geodesics of a pp-wave]
\label{prop:geod}
For a pp-wave expressed locally in the Brinkmann coordinates $(v,t,x^1,\dots,x^n)$ of Proposition \ref{prop:Brink}, its geodesic equations of motion are
\beqa
\label{eqn:geod*}
\left.\begin{array}{lcl}
\ddot{v} \!\!\!&=&\!\!\! \Big(2\frac{dV}{ds}-V_t\,\dot{t}\Big)\dot{t},\\
\ddot{t} \!\!\!&=&\!\!\! 0,\phantom{\Big(\Big)}\\
\ddot{x}^i \!\!\!&=&\!\!\! -V_i\,\dot{t}^2 \commas i = 1,\dots,n,\phantom{\Big(\Big)}
\end{array}\right\}
\eeqa
where $V_t \defeq \frac{\partial V}{\partial t}, V_i \defeq \frac{\partial V}{\partial x^i}$, and $s$ is the affine parameter.
\end{prop}

\begin{proof}
The only nonvanishing Christoffel symbols of \eqref{eqn:metric} are
\beqa
\label{eqn:Christ}
\Gamma_{it}^v = -V_{i} \comma \Gamma_{tt}^v = -V_t \comma \Gamma_{tt}^i = V_i,
\eeqa
from which the equations \eqref{eqn:geod*} are easily derived from the geodesic equations $\ddot{z}^{\alpha} + \sum_{\beta,\gamma}\Gamma^{\alpha}_{\beta\gamma}\dot{z}^\beta\dot{z}^\gamma = 0$.
\end{proof}

To see the relationship between the $(n+2)$-dimensional system \eqref{eqn:geod*} and the $n$-dimensional Hamiltonian system
\beqa
\label{eqn:0*}
\ddot{x} = -\nabla_{\!x}V \comma  x = (x^1,\dots,x^n) \in \RR^{n},
\eeqa
begin by observing that for any solution $\gamma(s) = (v(s),t(s),x^i(s))$ of \eqref{eqn:geod*} with affine parameter $s$, we must have $t(s) = as+b$ for some $a,b \in \RR$.  It follows that for $t$-initial data $t(0) = 0, \dot{t}(0) = 1$, the $\ddot{x}^i$-equations in \eqref{eqn:geod*} comprise precisely the Hamiltonian system \eqref{eqn:0*}, with $t$ equal to the affine parameter $s$. Moreover, as $v,\dot{v}$ make no appearance in any equation in \eqref{eqn:geod*}, the $\ddot{v}$-equation is completely determined by \eqref{eqn:0*}.  One therefore says that the solutions of \eqref{eqn:0*} ``lift" to the geodesics \eqref{eqn:geod*} of the $(n+2)$-dimensional pp-wave metric \eqref{eqn:metric}.  Let us formalize this now:

\begin{defn}[Hamiltonian/pp-wave correspondence]
\label{def:lift}
Let $\ddot{x} = -\nabla_{\!x}V$ be an $n$-dimensional Hamiltonian system with $C^2$ function $V(t,x)$ defined on an open connected set $I \times \mathcal{U} \subseteq \RR\times \RR^n$ with $t \in I$ and $x = (x^1,\dots,x^n) \in \mathcal{U}$.  Then the $(n+2)$-dimensional pp-wave metric
\beqa
\label{def:corr}
g \defeq 2dvdt  -2V(t,x)(dt)^2+ \sum_{i=1}^n (dx^i)^2,
\eeqa
defined on $(v,t,x) \in \RR\times I \times \mathcal{U} \subseteq \RR^{n+2}$, is the \emph{pp-wave lift of $V(t,x)$}.  Given any maximal solution $x(t)$ of $\ddot{x} = -\nabla_{\!x}V$, the geodesics of $g$ of the form $\gamma(t) = (v(t),t,x(t))$, with the same domain as $x(t)$, are the \emph{geodesic lifts of $x(t)$}.
\end{defn}

Note that because $\gamma$'s affine parameter $s = t$, its domain is indeed equal to $x(t)$'s.  (This is always possible by a linear rescaling unless $\dot{t}(s) = 0$, in which case $\gamma(s)$ is a straight line and there is no input from $\ddot{x} = -\nabla_{\!x}V$; we will therefore ignore this case henceforth.) We can easily obtain a one-to-one lift by fixing the starting point and the causal character of the geodesic\,---\,either spacelike, timelike, or lightlike\,---\,as follows (see \cite{eisenhart}, \cite[Proposition~3.1]{candela}, and \cite[Theorem~3.3]{minguzzi2} for generalizations):

\begin{prop}[Eisenhart lift]
\label{prop:3}
Let $V(t,x)$ be a $C^2$ function defined on an open connected set $I \times \mathcal{U} \subseteq \RR\times \RR^n$ and $(\RR\times I \times \mathcal{U},g)$ its pp-wave lift.  Then $x(t)$ is the maximal solution of $\ddot{x} = -\nabla_{\!x}V$ with initial data $(x_0,\dot{x}_0)$ at time $t_0$ if and only if $\gamma(t) = (v(t),t,x(t))$, with $v(t)$ satisfying $\ddot{v} = 2\frac{dV(t,x(t))}{dt}-\frac{\partial V}{\partial t}\big|_{(t,x(t))}$ and with initial data
\beqa
\label{eqn:conf0}
\gamma(t_0) \defeq (0,t_0,x_0) \commas \gamma'(t_0) \defeq \Big(\!\!-\!\frac{1}{2}\dot{x}_0^2 + V(t_0,x_0) + \vep_{\scalebox{0.6}{$\gamma$}},1,\dot{x}_0\Big),
\eeqa
is the geodesic lift of $x(t)$.  The constant $\vep_{\scalebox{0.6}{$\gamma$}}$ equals $0, -\frac{1}{2},\frac{1}{2}$ depending on whether $\gamma(t)$ is, respectively, lightlike, unit timelike, or unit spacelike.
\end{prop}

\begin{proof}
That the given $\gamma(t) = (v(t),t,x(t))$ solves \eqref{eqn:geod*} is clear; it is maximal if and only if $x(t)$ is maximal. In general, any such geodesic has initial velocity $\gamma'(t_0) = (\dot{v}_0,1,\dot{x}_0)$, but if we stipulate that $\gamma(t)$ is, say, lightlike, then
$$
g(\gamma'(t_0),\gamma'(t_0)) = 0 \imp 2\dot{v}_0 - 2V(t_0,x_0) + \sum_{i=1}^n(\dot{x}_0^i)^2 = 0.
$$
Since $\dot{x}_0^2 = \sum_{i=1}^n(\dot{x}_0^i)^2 $, solving for $\dot{v}_0$ yields precisely \eqref{eqn:conf0} with $\vep_{\scalebox{0.6}{$\gamma$}} = 0$.  The unit-timelike and -spacelike cases ($g(\gamma'(t_0),\gamma'(t_0)) = \mp1$) are similar.
\end{proof}

Fixing $\vep_{\scalebox{0.6}{$\gamma$}}$ thus gives a one-to-one lift.  The choice of $v_0 = 0$ in \eqref{eqn:conf0} is arbitrary and will play no role in what follows.  We close this section by reviewing a particularly important aspect of the Eisenhart Lift, namely, that it allows one to exploit the \emph{conformal invariance of lightlike geodesics} (see, e.g., \cite{Sanchez,minguzzi}).  To that end, let us recall the following facts regarding the conformal geometry of semi-Riemannian metrics.  Given a semi-Riemannian manifold $(M,g)$ with Levi-Civita connection $\nabla$, and any metric $e^{2f}g$ in its conformal class (with $f$ a $C^2$ function on $M$), an application of the Koszul formula shows that the Levi-Civita connection $\cdf{}{}$ of $e^{2f}g$ is related to $\nabla$ as follows: For any vector fields $X,Y \in \mathfrak{X}(M)$,
$$
\cdf{X}{Y} = \cd{X}{Y} + X(f)Y + Y(f)X - g(X,Y)\text{grad}_{\scalebox{0.6}{$g$}} f,
$$
where $\text{grad}_{\scalebox{0.6}{$g$}} f$ is the gradient of $f$ with respect to $g$ (in any choice of local coordinates, $\text{grad}_{\scalebox{0.6}{$g$}} f = g^{ij}f_i \partial_j$).  The key difference between the conformal properties of Riemannian vs. semi-Riemannian metrics is that the gradient term vanishes when $X = Y = N$ is a \emph{lightlike} vector field:
$$
g(N,N) = 0 \imp  \cdf{N}{N} = \cd{N}{N} + 2N(f)N.
$$
In particular, if $N$ has $g$-geodesic flow, $\cd{N}{N} = 0$, then $N$ will be \emph{pre}-geodesic with respect to $\cdf{}{}$; i.e., $\cdf{N}{N}$ will be proportional to $N$:
$$
g(N,N) = 0\ \ \ \text{and}\ \ \ \cd{N}{N} = 0 \imp  \cdf{N}{N} = 2N(f)N.
$$
Of course, this also holds for the covariant derivative along a curve $\gamma(t)$:
\beqa
\label{eqn:preg}
g(\gamma',\gamma') = 0\ \ \ \text{and}\ \ \ \cd{\gamma'}{\gamma'} = 0 \imp  \cdf{\gamma'}{\gamma'} = 2\frac{d(f \circ \gamma)}{dt} \gamma'.
\eeqa
The key property of pre-geodesics is that they can be rendered into geodesics by a reparametrization; this fact is well known, though we present its proof here for convenience, since we will call on this result in Theorem \ref{thm:1d} below:

\begin{prop}[Conformal invariance of lightlike geodesics]
\label{prop:clg}
Let $I \subseteq \RR$ be a connected interval and $\gamma\colon I \lra M$ a lightlike geodesic of a semi-Riemannian manifold $(M,g)$.  For any $t_0 \in I$ and any $C^2$ function $f$ on $M$, define the connected interval $\If \defeq \{\int_{t_0}^t  e^{2f(\gamma(r))}dr\,:\,t \in I\}$.  Then the $C^2$ curve $\gammaf\colon \If \lra M$ defined by
\beqa
\label{eqn:gamma1}
\gammaf(s) \defeq \gamma(\tau^{-1}(s)) \comma s = \tau(t) \defeq \int_{t_0}^t  e^{2f(\gamma(r))}dr
\eeqa
is a lightl$C^2$odesic of the conformal metric $e^{2f}g$.
\end{prop}

\begin{proof}
As $\tau$ is a diffeomorphism with $(\tau^{-1})'(s) = \frac{1}{\tau'(\tau^{-1}(s))} = e^{-2f(\gammaf(s))}$,
$$
\gammaf'(s) = e^{-2f(\gammaf(s))}\gamma'(\tau^{-1}(s)).
$$
This choice of $\tau$ ensures that
\beqa
\cdf{\gammaf'}{\gammaf'}\,\Big|_s \!\!&=&\!\! e^{-2f(\gammaf(s))}\cdf{\gamma'}{(e^{-2f(\gammaf(s))}\gamma')}\,\Big|_{\tau^{-1}(s)}\nonumber\\
&=&\!\! e^{-2f(\gammaf(s))}\Big[\gamma'(e^{-2f(\gammaf(s))})\gamma' + (e^{-2f(\gammaf(s))})\cdf{\gamma'}{\gamma'}\Big]\,\Big|_{\tau^{-1}(s)}\nonumber\\
&\overset{\eqref{eqn:preg}}{=}&\!\! e^{-2f(\gammaf(s))}\Big[\underbrace{\gamma'(e^{-2f(\gamma(t))}) + 2(e^{-2f(\gamma(t))})\frac{d(f \circ \gamma)}{dt}}_{0}\Big]\gamma'(t)\nonumber\\
&=&\!\! 0,\nonumber
\eeqa
which completes the proof.
\end{proof}

As shown in \cite[Theorem~3.2]{minguzzi2}, the relevance of this fact for \eqref{eqn:0} is the following: Any solution of the $n$-dimensional Hamiltonian system $\ddot{x} = -\nabla_{\!x}V,$
when it is lifted to the lightlike geodesic \eqref{eqn:conf0} and reparametrized as in \eqref{eqn:gamma1}, \emph{is also a solution of an infinite ``conformal class" of non-diffeomorphic ODEs}.  This conformal class is written in eqn.~(34) of \cite{minguzzi2} as stationary points of an action (of which \eqref{eqn:0} and pp-waves are a special case).  Let us present eqn.~(34) of \cite{minguzzi2} here in terms of these conformal ODEs explicitly:

\begin{thm}[Conformal class of a Hamiltonian system]
\label{thm:1d}
Let $g$ be an $(n+2)$-dimensional pp-wave expressed in the Brinkmann coordinates $(v,t,x^1,\dots,x^n)$ of Proposition \ref{prop:Brink}.  For any $C^2$ function $f(t,x^1,\dots,x^n)$, the geodesic equations of motion of the conformal metric $e^{2f}g$ are
\beqa
\label{eqn:geod2*}
\left.\begin{array}{lcl}
\ddot{v} \!\!\!&=&\!\!\! \Big(2\frac{dV}{ds}-V_t\,\dot{t}\Big)\dot{t} -2Vf_t\,\dot{t}^2 - \sum_{i=1}^n\big(2f_i\dot{v} - f_t\dot{x}^i\big)\dot{x}^i,\phantom{\Big(\Big)}\\
\ddot{t} \!\!\!&=&\!\!\! -2\frac{df}{ds}\dot{t},\phantom{\Big(\Big)}\\
\ddot{x}^i \!\!\!&=&\!\!\! -V_i\,\dot{t}^2 +2(\dot{v}-V\dot{t})f_i\dot{t} - \big((\dot{x}^i)^2-\sum_{j\neq i}^n (\dot{x}^j)^2\big)f_i\phantom{\Big(\Big)}\\
&& \hspace{1in}-2\dot{x}^i\sum_{j\neq i}^n f_j\dot{x}^j -2f_t\,\dot{t}\dot{x}^i\commas i = 1,\dots,n.\phantom{\Big(\Big)}
\end{array}\right\}
\eeqa
Thus for every solution $x(t)$ of
$
\ddot{x}^i = -\nabla_{\!x}V,
$
its corresponding lightlike geodesic lift $\gamma(t) = (v(t),t,x(t))$ is, after the reparametrization \eqref{eqn:gamma1}, also a solution of \eqref{eqn:geod2*}.
\end{thm}

\begin{proof}
In the coordinate basis $\{\partial_v,\partial_t,\partial_1,\dots,\partial_n\}$, the components of the conformal metric $e^{2f}g$ and its inverse are, respectively,
$$
e^{2f}\begin{pmatrix}
0 & 1 & 0 & \cdots & 0\\
1 & -2V(t,x) & 0 & \cdots & 0\\
0 & 0 & 1 & \cdots &0\\
\vdots & \vdots & \vdots & \ddots &\vdots\\
0 & 0 & 0 & \cdots & 1
\end{pmatrix} \comma e^{-2f}\begin{pmatrix}
2V(t,x) & 1 & 0 & \cdots & 0\\
1 & 0 & 0 & \cdots & 0\\
0 & 0 & 1 & \cdots &0\\
\vdots & \vdots & \vdots & \ddots &\vdots\\
0 & 0 & 0 & \cdots & 1
\end{pmatrix}\cdot
$$
Together, these yield the following nonvanishing Christoffel symbols:
\beqa
\Gamma_{vi}^v = \Gamma_{ti}^t = \Gamma_{ii}^i = -\Gamma_{jj}^i = -\Gamma_{vt}^i = \Gamma_{ji}^j = f_i \!\!\!\!\!&\commas&\!\!\!\!\! -\Gamma_{ii}^v = \frac{1}{2}\Gamma_{tt}^t = \Gamma_{it}^i = f_t,\nonumber\\
\Gamma_{tt}^v = 2Vf_t - V_t \comma \Gamma_{ti}^v = -V_i \!\!\!&\comma&\!\!\!\Gamma_{tt}^i = 2Vf_i+V_i.\nonumber  
\eeqa
From here \eqref{eqn:geod2*} is straightforwardly derived.  The statement regarding solutions of the $n$-dimensional Hamiltonian follows from Proposition \ref{prop:clg}.
\end{proof}

E.g., for $n=1$, $f=f(x)$, and $V=V(x)$, \eqref{eqn:geod2*} reduces to the system
$$
\ddot{v} = 2\frac{dV}{ds}\dot{t} - 2\frac{df}{ds}\dot{v} \commas \ddot{x} = -V_x\,\dot{t}^2 - (2V\dot{t}^2 +\dot{x}^2-2\dot{t}\dot{v})f_x \commas  \dot{t} = \dot{t}_0e^{\text{$-2\int_{s_0}^s\!\frac{df}{dr}dr$}}.
$$
By Theorem \ref{thm:1d}, solutions of $\ddot{x} = -V_x$ thus lift to (lightlike) solutions of this system.  In fact Theorem \ref{thm:1d} is true even when $f_v \neq 0$, although \eqref{eqn:geod2*} is more elaborate in this case.  Note also that, while $\nabla$ and $\cdf{}{}$ arise from conformal metrics, they will be projectively equivalent\,---\,i.e., up to reparametrization, \emph{all} their geodesics will coincide\,---\,if and only if $\text{grad}_{\scalebox{0.6}{$g$}} f = 0$ (see \cite[Theorem~1.2.3]{taber}).  Finally, see \cite{duval_gibbons} for a physical application of the conformal invariance of lightlike geodesics of pp-waves, and \cite{gordon2,ong} for another conformal approach to \eqref{eqn:0} via the (Riemannian) ``Jacobi metric."

\section{Curvature and the stability of Hamiltonian trajectories}
\label{sec:conjugate}
Let us now consider \emph{conjugate points} along $\gamma(t) = (v(t),t,x(t))$; i.e., those points along $\gamma(t)$ at which ``neighboring" geodesics accumulate.  Projecting onto $x(t)$ then gives conditions on $V$ under which trajectories of $\ddot{x} = -\nabla_{\!x}V$ will accumulate.  Our treatment follows \cite{flores2}\,---\,Theorems \ref{thm:conj_1} and \ref{thm:last_conj} below complement \cite[Proposition~6.4, Remark~6.5]{flores2}. (See also \cite{casetti2}, where conjugate points are applied to chaos theory.)

\begin{defn}[Conjugate point]
\label{def:conj}
Let $(M,g)$ be a semi-Riemannian manifold and $\gamma\colon [0,b] \lra M$ a geodesic of any causal character. A $C^2$ vector field $J(t)$ along $\gamma(t)$ is a \emph{Jacobi field} if \beqa
\label{eqn:J1}
J''+R_{\scalebox{0.6}{$g$}}(J,\gamma')\gamma'\,\Big|_t = 0,
\eeqa
where $R_{\scalebox{0.6}{$g$}}$ is the curvature endomorphism of $g$.  The points $\gamma(0)$ and $\gamma(b)$ are \emph{conjugate} if there is a nontrivial Jacobi field $J(t)$ along $\gamma(t)$ satisfying $J(0) = J(b) = 0$. Their \emph{multiplicity} is the dimension of the linear subspace of Jacobi fields $J(t)$ along $\gamma(t)$ satisfying $J(0) = J(b) = 0$.
\end{defn}

If $\gamma(b)$ is conjugate to $\gamma(0) = (v(0),0,x(0))$, then there is a one-parameter family $\gamma_s(t)$ of geodesics (with $\gamma_0(t) = \gamma(t)$) all starting at $\gamma(0)$ and ``accumulating" at $\gamma(b)$\,---\,i.e., the corresponding Jacobi field $J(t)$, which equals the infinitesimal variation $\frac{\partial \gamma_s(t)}{\partial s}\big|_{s=0}$ of $\gamma_s(t)$ from $\gamma(t)$, will satisfy $J(b) = 0$ (see Lemma \ref{prop:check} below, and note that $\gamma_s(b)$ needn't actually equal $\gamma(b)$).  If $g$ is a pp-wave, then projecting Definition \ref{def:conj} down to $x(t)$ leads naturally to a notion of ``accumulation point" for \eqref{eqn:0}. This correspondence was formulated and studied in \cite{flores2} (in fact for manifolds more general than pp-waves).  The following is a special case of \cite[Definition~6.1]{flores2}:

\begin{defn}[Conjugate point of a Hamiltonian trajectory]
\label{def:conj2}
Let $V(t,x)$ be a $C^2$ function, $x\colon [0,b]\lra \RR^n$ a nonconstant solution of $\ddot{x} = -\nabla_{\!x}V$, and $\gamma(t) = (v(t),t,x(t))$ its geodesic lift (of any causal character).  Then \emph{$x(b)$ is conjugate to $x(0)$} if $\gamma(b)$ is conjugate to $\gamma(0)$ as in Definition \ref{def:conj}.
\end{defn}

Our first task is to show how conjugate points along $x(t)$ give rise to a one-parameter family of \emph{distinct} solutions of $\ddot{x} = -\nabla_{\!x}V$ all starting at $x(0)$ and accumulating at $x(b)$, and to make precise the notion of ``accumulation point" just described:

\begin{lemma}
\label{prop:check}
Let $x\colon [0,b] \lra \RR^n$ be a nonconstant solution of $\ddot{x} = -\nabla_{\!x}V$.  If $x(b)$ is conjugate to $x(0)$, then there is a one-parameter family of distinct solutions of $\ddot{x} = -\nabla_{\!x}V$ starting at $x(0)$ and accumulating at $x(b)$.
\end{lemma}

\begin{proof}
By Definition \ref{def:conj2}, the geodesic lift $\gamma(t) = (v(t),t,x(t))$ of $x(t)$ has $\gamma(b)$ conjugate to $\gamma(0)$ (the value of $\vep_{\scalebox{0.6}{$\gamma$}}$ in \eqref{eqn:conf0} will not matter in what follows).  Thus there is a nontrivial Jacobi field $J(t)$ along $\gamma(t)$ satisfying $J(0) = J(b) = 0$.  That $J(t)$ is nontrivial implies that $J'(0) \neq 0$; that it vanishes at two points implies that it must also be orthogonal to $\gamma'(t)$ (by a similar reasoning as in \eqref{eqn:zeroJ} below).  Writing $$J(t) = (J^v(t),J^t(t),J^1(t),\dots,J^{n}(t))$$ in the Brinkmann coordinates of Proposition \ref{prop:Brink}, let us start by observing that the $t$-component $J^t(t) = g(J,\partial_v)\big|_{\gamma(t)}$ must be linear in $t$ because $\partial_v$ is a parallel vector field:
\beqa
\label{eqn:zeroJ}
\ddot{J}^t(t) =g(J'',\partial_v)\Big|_{\gamma(t)} \overset{\eqref{eqn:J1}}{=} -\text{Rm}_{\scalebox{0.6}{$g$}}(J,\gamma',\gamma',\partial_v)\Big|_{\gamma(t)} = 0.
\eeqa
But if $J(t)$ is to be zero at two distinct points, then we must have $J^t(t) = 0$.  Next, the orthogonality condition, $g(J,\gamma')\big|_{\gamma(t)} = 0$, implies that $J^v(t) = -\!\sum_{i=1}^n J^i(t)\dot{x}^i(t)$, so that $J(t)$ takes the form
$$
J(t) = -\Big(\sum_{i=1}^{n}J^i(t)\dot{x}^i(t)\Big)\partial_v + \sum_{i=1}^{n} J^i(t) \partial_i\,\Big|_{\gamma(t)}.
$$
Taking now the covariant derivative of $J(t)$ along $\gamma(t)$ (and recalling \eqref{eqn:Christ}), it is straightforward to verify that
$$
J'(t) = \bigg[\!\!-\!\frac{d}{dt}\Big(\sum_{i=1}^n J^i(t)\dot{x}^i(t)\Big) - \sum_{i=1}^n J^i(t)V_i(t,x(t))\bigg]\partial_v + \sum_{i=1}^n\dot{J}^i(t)\partial_i\,\Big|_{\gamma(t)}.
$$
At $t=0$, this simplifies because each $J^i(0) = 0$:
\beqa
\label{eqn:J'}
J'(0) = -\Big(\sum_{i=1}^n \dot{J}^i(0)\dot{x}^i(0)\Big)\partial_v + \sum_{i=1}^n\dot{J}^i(0)\partial_i\,\Big|_{\gamma(0)} \neq 0.
\eeqa
Now consider the one-parameter family of geodesics $\gamma_s(t)$ defined by
\beqa
\label{eqn:VVF}
\gamma_s(t) \defeq \text{exp}_{\gamma(0)}\big(t(\gamma'(0)+sJ'(0))\big) \comma s \in (-\vep,\vep)\commass t \in [0,b],
\eeqa
where $\text{exp}_{\gamma(0)}$ is the exponential map and $\vep$ is small enough; note that $\gamma_0(t) = \gamma(t)$ and that each $\gamma_s(t)$ starts at $\gamma(0) = (v(0),0,x(0))$ in the direction
$$
\underbrace{\,\gamma'(0)+sJ'(0)\,}_{\text{$\gamma_s'(0)$}} \overset{\eqref{eqn:J'}}{=} \Big(\dot{v}(0)-s\!\sum_{i=1}^n \dot{J}^i(0)\dot{x}^i(0)\Big)\partial_v+\partial_t + \sum_{i=1}^n\big(\dot{x}^i(0)+s\dot{J}^i(0)\big)\partial_i\,\Big|_{\gamma(t)}.
$$
Furthermore, $\gamma_s(t) = (v_s(t),t,x_s(t))$ (though $g(\gamma_s',\gamma_s')$ may not equal $\vep_{\scalebox{0.6}{$\gamma$}}$), hence is a lift of a solution $x_s(t)$ of $\ddot{x} = -\nabla_{\!x}V$ with $x_s(0) = x(0)$ and $\dot{x}_s(0) = \dot{x}(0)+s\dot{J}(0)$.  As $\frac{\partial \gamma_s(t)}{\partial s}\big|_{s=0} = J(t)$ (cf.~\cite[Proposition~10,~p.~271]{o1983}), $J(b) = 0$ indicates the accumulation of the $\gamma_s(b)$'s, and thus their $x_s(b)$'s.  By \eqref{eqn:J'}, some $\dot{J}^i(0) \neq 0$, hence each $\dot{x}_s(0)$ (and thus each $x_s(t)$) is distinct.
\end{proof}

With this established, let us now recall how conjugate points behave with respect to conformal transformations of $g$; our treatment follows \cite{minguzzi}.  First, observe that $\gamma'(t)$ itself is a Jacobi field along $\gamma(t)$; as a consequence, for any Jacobi field $J(t)$ and any $C^2$ function $f(t)$ along $\gamma(t)$,
\beqa
\label{eqn:J2}
(J+f\gamma')'' + \underbrace{\,R(J,\gamma')\gamma' + f\cancelto{0}{R(\gamma',\gamma')\gamma'}\,}_{\text{$R(J+f\gamma',\gamma')\gamma'$}} = \underbrace{\,(J''+ R(J,\gamma')\gamma')\,}_{0} + f''\gamma'.
\eeqa
For lightlike geodesics, \eqref{eqn:J2} has the following consequence.  First, consider the subspace $\mathcal{N}(\gamma)$ of $C^2$ vector fields along $\gamma(t)$ that are orthogonal to $\gamma(t)$; as mentioned in Lemma \ref{prop:check}, this includes all Jacobi fields that vanish at two points along $\gamma(t)$.  Now, if $\gamma(t)$ is a \emph{lightlike} geodesic, then $\gamma'(t) \in \mathcal{N}(\gamma)$ also.  Thus we can consider the quotient space $\mathcal{N}(\gamma)/\gamma'$, whose equivalence classes are defined by $[J_1] = [J_2]$ if and only if $J_2 = J_1 +f\gamma'$
for some $C^2$ function $f(t)$ along $\gamma(t)$.  Furthermore, we can take derivatives of equivalence classes along $\gamma(t)$ by defining, for any $X \in \mathcal{N}(\gamma)$,
$$
[X]' \defeq [X'] \comma \cd{\gamma'}{[X]} \defeq [\cd{\gamma'}{X}],
$$
both of which are well defined when $\gamma(t)$ is a lightlike geodesic.  As a consequence, the Jacobi equation \eqref{eqn:J1}, too, descends onto $\mathcal{N}(\gamma)/\gamma'$.  Indeed,  \eqref{eqn:J2} shows that if $J_1$ satisfies the Jacobi equation,
$
J_1'' + R(J_1,\gamma')\gamma'\,\big|_t = 0,
$
then every representative $J_2$ of its equivalence class $[J_1] \in \mathcal{N}(\gamma)/\gamma'$ satisfies the following ``quotient Jacobi equation":
$$
[J_2]'' + [R(J_2,\gamma')\gamma'] = [J_2'' + R(J_2,\gamma')\gamma'] \overset{\eqref{eqn:J2}}{=} [f''\gamma'] = 0.
$$
This suggests that, along the lightlike geodesic $\gamma(t)$, the Jacobi equation effectively only sees the conformal structure.  The following well known result confirms that this is indeed the case:

\begin{prop}[Conformal invariance of lightlike conjugate points]
\label{prop:clg2}
Let $\gamma\colon[0,b] \lra M$ be a lightlike geodesic of a semi-Riemannian manifold $(M,g)$ and $J(t)$ a Jacobi field along $\gamma(t)$ satisfying $J(0) = J(b) = 0$.  For any conformal metric $e^{2f}g$, let $\gammaf\colon \If \lra M$ denote the reparametrization of $\gamma$ as a geodesic of $e^{2f}g$, as in \eqref{eqn:gamma1}.  Then $\Jf \colon \If \lra M$  defined by $\Jf \defeq J(\tau^{-1}(s))$ is a Jacobi field along $\gammaf$ satisfying $\Jf(0) = \Jf(\tau(b)) = 0$.  In particular, the set of conjugate points of a lightlike geodesic, including their multiplicities, is a conformal invariant.
\end{prop}

\begin{proof}
We refer the reader to \cite[Theorem~2.36]{minguzzi} for a proof.
\end{proof}

With this in hand, we can add to the conclusions of Theorem \ref{thm:1d}:

\begin{cor}
\label{cor:main}
In Theorem \ref{thm:1d}, the conjugate points of the lightlike geodesic lifts, together with their multiplicities, are the same for the system of ODEs \eqref{eqn:geod*} and \eqref{eqn:geod2*}.
\end{cor}

We now present, in Theorems \ref{thm:conj_1} and \ref{thm:last_conj}, two cases in which conjugate points exist; these complement \cite[Proposition~6.4, Remark~6.5]{flores2}, which gave a condition for their presence/absence in terms of the convexity of $V$'s Hessian and the sign of the pp-wave's sectional curvature.  Here we focus on the trace of $V$'s Hessian, $\Delta_xV$.  Our first case assumes that the latter has a positive lower bound:

\begin{thm}
\label{thm:conj_1}
Let $V(t,x)$ be a $C^2$ function and $x\colon[0,b]\lra \RR^n$ a nonconstant solution of $\ddot{x} = -\nabla_{\!x}V$.  If $\Delta_xV\big|_{(t,x(t))} \geq (n-1)\frac{\pi^2}{b^2}$, then there is a one-parameter family of distinct solutions of $\ddot{x} = -\nabla_{\!x}V$ beginning at $x(0)$ and accumulating at some later point $x(t_1)$, where $0 < t_1 \leq b$.
\end{thm}

\begin{proof}
Let $\gamma(t) = (v(t),t,x(t))$ be the unit-timelike geodesic lift of $x(t)$; i.e., with $\vep_{\scalebox{0.6}{$\gamma$}} = -\frac{1}{2}$ in \eqref{eqn:conf0}.  Because $\gamma(t)$'s domain is also $[0,b]$, its length is $\int_0^b \sqrt{-g(\gamma',\gamma')}\,dt = \int_0^b dt = b$. Furthermore, the Ricci tensor $\text{Ric}_{\scalebox{0.6}{$g$}}$ of the pp-wave lift $g$, evaluated at each $\gamma'(t) = \dot{v}\partial_v+\partial_t+\sum_{i=1}^n\dot{x}^i\partial_i\big|_{\gamma(t)}$, is
\beqa
\label{eqn:Ricci0}
\text{Ric}_{\scalebox{0.6}{$g$}}(\gamma',\gamma')\Big|_{\gamma(t)} = \text{Ric}_{\scalebox{0.6}{$g$}}(\partial_t,\partial_t)\Big|_{\gamma(t)} = \Delta_x V\Big|_{(t,x(t))} \geq (n-1)\frac{\pi^2}{b^2}\cdot
\eeqa
(In the Brinkmann coordinates \eqref{eqn:metric}, $\text{Ric}_{\scalebox{0.6}{$g$}}(\partial_t,\partial_t)$ is the only nonvanishing component of the Ricci tensor of $g$; that this component equals $\Delta_xV$ can be confirmed, e.g., in \cite{Leistner}, or via \eqref{eqn:Christ}.)  Via \cite[Lemma~23,~p.~278]{o1983}, \eqref{eqn:Ricci0} guarantees the existence of a point $\gamma(t_1)$ conjugate to $\gamma(0)$, with $0 < t_1 \leq b$.  By Definition \ref{def:conj2}, $x(t_1)$ is therefore conjugate to $x(0)$; by Lemma \ref{prop:check}, there is thus a one-parameter family of distinct solutions of $\ddot{x} = -\nabla_{\!x}V$ beginning at $x(0)$ and accumulating at $x(t_1)$.
\end{proof}

Theorem \ref{thm:conj_1} required the Laplacian $\Delta_x V$ to be bounded away from zero. In fact one can still obtain conjugate points by assuming the function to be merely subharmonic along $x(t)$, $\Delta_x V\big|_{(t,x(t))} \geq 0$, provided $x(t)$ can now be defined for all time.  Here, then, is the second case:
\begin{thm}
\label{thm:last_conj}
Let $V(t,x)$ be a $C^2$ function and $x(t)$ a maximal, nonconstant solution of $\ddot{x} = -\nabla_{\!x}V$ satisfying the following properties:
\begin{enumerate}[leftmargin=.4in]
\item[i.] There is a time $t_0$ at which $\dot{x}^i(t_0) = 0$ for some $i=1,\dots,n$,
\item[ii.] $V_{ij}\big|_{(t_0,x(t_0))} \neq 0$ for some $j=1,\dots,n$,
\item[iii.] $\Delta_x V\big|_{(t,x(t))} \geq 0$.
\end{enumerate}
If $x(t)$ is defined for all time, then there is a one-parameter family of distinct solutions of $\ddot{x} = -\nabla_{\!x}V$ starting at some point $x(t_1)$ and accumulating at some later point $x(t_2)$.
\end{thm}

\begin{proof}
Let $\gamma(t) = (v(t),t,x(t))$ be the lightlike geodesic lift of $x(t)$; i.e., with $\vep_{\scalebox{0.6}{$\gamma$}} = 0$ in \eqref{eqn:conf0}.  Because $\dot{\gamma}^i(t_0) = \dot{x}^i(t_0) = 0$, the vector $\partial_i|_{\gamma(t_0)}$ is orthogonal to $\gamma'(t_0)$, and
$$
\text{Rm}_{\scalebox{0.6}{$g$}}(\partial_i,\gamma',\gamma',\partial_j)\Big|_{\gamma(t_0)} = \text{Rm}_{\scalebox{0.6}{$g$}}(\partial_i,\partial_t,\partial_t,\partial_j)\Big|_{\gamma(t_0)} = V_{ij}\Big|_{(t_0,x(t_0))} \neq 0.
$$
(Once again, this computation can be confirmed via \cite{Leistner} or \eqref{eqn:Christ}.) Furthermore, $\text{Ric}_{\scalebox{0.6}{$g$}}(\gamma',\gamma')\big|_{\gamma(t)} = \Delta_xV\big|_{(t,x(t))} \geq 0$, while $\gamma(t)$ is defined for all time because its projection $x(t)$ is assumed to be.  By \cite[Proposition~12.17, p.~444]{beem}, these conditions come together to guarantee that $\gamma(t)$ will have a pair of conjugate points.  The result now follows once again from Lemma \ref{prop:check}. (Note that Proposition \ref{prop:clg2} also applies here, to yield conjugate points along lightlike solutions of the system \eqref{eqn:geod2*}.)
\end{proof}

\section{Eisenhart Lifts of Second-order complex ODEs}
\label{sec:comp}
For our next application, we move to the complex plane $\mathbb{C}$.  Let $F\colon \mathbb{C} \lra \mathbb{C}$ be a holomorphic function and consider the complex second-order ODE
\beqa
\label{eqn:Code}
\ddot{z} = F(z(t)) \comma z = x+iy.
\eeqa
We would like to lift this ODE as with the Eisenhart lift, namely, as the geodesic equations of motion of a (real) pp-wave.  As we now show, this can be done, but it will require four-dimensional ``pp-waves" of signature $(-\!-\!++)$.  Before defining such semi-Riemannian metrics, let us recast $F$:
\begin{lemma}
Let $F\colon \mathbb{C} \lra \mathbb{C}$ be a holomorphic function.  Then there exists a harmonic function $V(x,y)$ on $\RR^2$ such that
\beqa
\label{eqn:fH}
F(z) = V_x +i(-V_y).
\eeqa
\end{lemma}

\begin{proof}
Let us write $F$ as
$
F(z) = u(x,y) + iv(x,y),$
where, via the Cauchy-Riemann equations, $u$ and $v$ satisfy $u_x = v_y$ and $u_y=-v_x$.  To express $F$ as in \eqref{eqn:fH}, we need $V$ to satisfy $V_x = u, V_y = -v$, which is the case if and only if $$u_y = V_{xy} = V_{yx} = -v_x;$$ this the second of the Cauchy-Riemann equations.  The first, $u_x = v_y$, then guarantees that $V$ must be harmonic.
\end{proof}

As a consequence, the complex ODE \eqref{eqn:Code} can be recast as
\beqa
\label{eqn:Code2}
\ddot{x} = V_x(x(t),y(t)) \comma \ddot{y} = -V_y(x(t),y(t)).
\eeqa
It turns out that these lift to geodesics of the following metric:

\begin{prop}[Split-signature pp-wave]
On $\RR^4 = \{(v,t,x,y)\}$, consider the semi-Riemannian metric $g$ of signature $(-\!-\!++)$ defined by
\beqa
\label{eqn:metric2}
\gS \defeq 2dvdt - 2V(t,x,y)(dt)^2 - (dx)^2 + (dy)^2,
\eeqa
where $V$ is any $C^2$ function independent of $v$.  Then the geodesic equations of motion of $(\RR^4,\gS)$ are
\beqa
\label{eqn:geod2**}
\left.\begin{array}{lcl}
\ddot{v} \!\!\!&=&\!\!\! \Big(2\frac{dV}{ds}-V_t\,\dot{t}\Big)\dot{t}\\
\ddot{t} \!\!\!&=&\!\!\! 0,\phantom{\Big(\Big)}\\
\ddot{x} \!\!\!&=&\!\!\! V_x\,\dot{t}^2 ,\phantom{\Big(\Big)}\\
\ddot{y} \!\!\!&=&\!\!\! -V_y\,\dot{t}^2 ,\phantom{\Big(\Big)}
\end{array}\right\}
\eeqa
where $V_t \defeq \frac{\partial V}{\partial t}, V_i \defeq \frac{\partial V}{\partial x^i}$, and $s$ is the affine parameter.
\end{prop}

\begin{proof}
This is similar to Proposition \ref{prop:geod}; indeed, the nonvanishing Christoffel symbols of \eqref{eqn:metric2} are
$$
\Gamma_{it}^v = -V_{i} \comma \Gamma_{tt}^v = -V_t \comma \Gamma_{tt}^x = -V_x \comma \Gamma_{tt}^y = V_y,
$$
the only difference being that $\Gamma_{tt}^x = -V_x$ instead of $\Gamma_{tt}^x = V_x$, as would have been the case for a Lorentzian-signature pp-wave.
\end{proof}

(Such higher-signature metrics have also been considered for physical reasons, in \cite{gibbons2}.)  The relationship between \eqref{eqn:geod2**} and \eqref{eqn:Code2} is identical to that between \eqref{eqn:geod*} and \eqref{eqn:0*}, and thus we arrive at the analogue of Proposition \ref{prop:3}:

\begin{prop}[Complex version of Eisenhart lift]
\label{prop:3*}
Let $V(x,y)$ be a harmonic function defined on an open connected set $\mathcal{U} \subseteq \RR^2$ and let $(\RR^2 \times \mathcal{U},\gS)$ be the corresponding semi-Riemannian metric $\gS$ given by \eqref{eqn:metric2}.  Set $F \defeq V_x +i(-V_y)$.  Then $z(t) = x(t)+iy(t)$ is a maximal solution of the complex ODE $\ddot{z} = F(z(t))$ with initial data $(z_0,\dot{z}_0)$ at $t_0 \in \RR$ if and only if $\gamma(t) = (v(t),t,x(t),y(t))$, with $v(t)$ satisfying $\ddot{v} = 2\frac{dV(x(t),y(t))}{dt}$ and with initial data
\beqa
\label{eqn:conf00}
\gamma(t_0) \defeq (0,t_0,z_0) \commas \gamma'(t_0) \defeq \Big(\!\!-\!\frac{1}{2}\dot{z}_0^2 + V(z_0) + \vep_{\scalebox{0.6}{$\gamma$}},1,\dot{z}_0\Big)\cdot
\eeqa
is the geodesic lift of $x(t)$.  The constant $\vep_{\scalebox{0.6}{$\gamma$}}$ equals $0, -\frac{1}{2},\frac{1}{2}$ depending on whether $\gamma(t)$ is, respectively, lightlike, unit timelike, or unit spacelike.
If $F$ is entire but not linear in $x,y$, then there are solutions of $\ddot{z} = F(z(t))$ and geodesics of $\gS$ that blow up in finite time.
\end{prop}

\begin{proof}
The proof is identical to that of Proposition \ref{prop:3}.  The statement about blow-ups is due to \cite[Corollary~7.4]{forst}. 
\end{proof}

We also have the analogue of Theorem \ref{thm:1d}:

\begin{thm}[Conformal class of the complex ODE $\ddot{z} = F(z(t))$]
\label{thm:main2}
Let $V(x,y)$ be a harmonic function defined on an open connected set $\mathcal{U} \subseteq \RR^2$ and let $(\RR^2 \times \mathcal{U},\gS)$ be the corresponding semi-Riemannian metric \emph{$\gS$} given by \eqref{eqn:metric2}.  Set $F \defeq V_x +i(-V_y)$.  By Proposition \ref{prop:3*}, there is a bijection between solutions $(z(t),z_0,\dot{z}_0)$ of $\ddot{z} = F(z(t))$ and lightlike geodesics $\gamma(t)$ of \emph{$\gS$} with initial data \eqref{eqn:conf00}. By Proposition \ref{prop:clg} and this bijection, each solution $z(t)$ is, after the reparametrization \eqref{eqn:gamma1}, also a solution of the geodesic equations of motion of the conformal metric \emph{$e^{2f}\gS$} for any $C^2$ function $f$ on $\mathcal{U}$.  By Proposition \ref{prop:clg2}, this also includes all conjugate points and their multiplicities.
\end{thm}

Let us close by mentioning that\,---\,as was already noted in \cite{eisenhart} (see also \cite{casetti2})\,---\,if we consider more general pp-wave metrics of the form
$$
g \defeq 2dvdt -2V(t,x^1,\dots,x^n)(dt)^2+ \sum_{i=1}^n a_i(dx^i)^2 \commas a_1,\dots,a_n \in \RR\backslash\{0\},
$$
then we see that its geodesics are lifts of \emph{anisotropic} Hamiltonian systems; i.e., of the form $
\ddot{x} = -\big(a_1V_1(t,x),\dots,a_nV_n(t,x)\big)$ with $a_1,\dots,a_n \in \RR\backslash\{0\}.$
By the analogue of Theorem \ref{thm:main2}, these also ``unfold" into an infinite family of ``conformal ODEs" with shared solutions and accumulation points.

\section{Riemannian lifts dual to Eisenhart lifts}
\label{sec:Riem}
Finally, we show that, for Hamiltonian systems with \emph{nonnegative, time-independent} functions, there is a Riemannian version of the Eisenhart lift that is ``dual" to it in a sense that we make precise now.  Thus, fix a $C^2$ function $\widetilde{V}(x) \geq 0$ and consider
$
\ddot{x} = -\nabla_{\!x}\widetilde{V}.
$
We claim that solutions of this $n$-dimensional Hamiltonian system lift to the geodesics of the $(n+2)$-dimensional \emph{Riemannian} metric $\gR$ defined by
$$
\gR \defeq g + 2T^{\flat} \otimes T^{\flat} \comma T \defeq \Big(\!\!-\!V+\frac{1}{2}\Big)\partial_v - \partial_t \comma V(t,x) \in \RR,
$$
where $V(t,x)$ is a $C^2$ function to be determined, and $g$ is a pp-wave in Brinkmann coordinates $(v,t,x^1,\dots,x^n)$:
\beqa
\label{eqn:Riem}
((\gR)_{ij}) = \begin{pmatrix}2 & -2V& 0 & \cdots &0\\ -2V & \frac{1}{2}(1+4V^2)& 0& \cdots & 0 \\ 0 & 0 &1 & \cdots & 0\\ \vdots & \vdots & \vdots & \ddots & \vdots\\0 &0&0&\cdots & 1\end{pmatrix}\cdot
\eeqa

Such metrics ``dual to pp-waves" were considered in \cite{aazami} and shown to be \emph{almost-K\"ahler} metrics; \cite{lejmi} then showed that included among them are distinguished almost-K\"ahler metrics, namely, so called \emph{extremal} and \emph{second-Chern–Einstein} examples.  Our interest here is solely in their geodesics, which were analyzed in \cite[Proposition~1]{aazami}.  In the coordinates $(v,t,x^1,\dots,x^n)$, and with $s$ the affine parameter, they simplify to
\beqa
\left.\begin{array}{ccc}
\ddot{v} \!\!&=&\!\! \frac{d}{ds}(V\dot{t}),\phantom{\Big(\Big)}\\
\ddot{t} \!\!&=&\!\! 2c_0\Big(\!\frac{dV}{ds} - V_t\,\dot{t}\Big),\phantom{\Big(\Big)}\\
\ddot{x}^i \!\!&=&\!\! -c_0V_i\,\dot{t},\phantom{\Big(\Big)}\\
\end{array}\right\}\label{eqn:Kahler}
\eeqa
where the constant $c_0$ is the constant of the motion obtained via the Killing vector field $\partial_v$,
\beqa
\label{eqn:v_end}
c_0 \defeq \gR(\partial_v,\gamma'(s)) = 2\dot{v}-2V\dot{t}\,\Big|_s,
\eeqa
with the geodesic $\gamma(s) = (v(s),t(s),x(s))$ satisfying \eqref{eqn:Kahler}.  In fact $c_0$ can also be obtained by integrating $\ddot{v}$; with initial data $(v_0,t_0,x_0)$ and $(\dot{v}_0,\dot{t}_0,\dot{x}_0)$,
$$
\dot{v} - \dot{v}_0 = V\dot{t} - V(t_0,x_0)\dot{t}_0 \imp \dot{v} - V\dot{t}\,\Big|_s = \dot{v}_0 - V(t_0,x_0)\dot{t}_0 \overset{\eqref{eqn:v_end}}{=} \frac{c_0}{2}\cdot
$$
Given that $\ddot{x} = -c_0\,\dot{t}\,\nabla_{\!x}V$ in \eqref{eqn:Kahler} resembles a Hamiltonian system, we will be interested in the case when $c_0 \neq 0$.  Next, comparing \eqref{eqn:Kahler} with \eqref{eqn:geod*}, observe that the $\ddot{v}$-equation is once again redundant (there being no appearance of $v$ or $\dot{v}$ terms), but that now the $t$-coordinate is no longer linearly related to the affine parameter $s$, as it was in \eqref{eqn:geod*}:
$$
\dot{t} = \dot{t}_0 + 2c_0\!\int_{s_0}^s\underbrace{\Big(\frac{dV}{ds} - V_t\,\dot{t}\Big)}_{\text{$\nabla_{\!x}V\cdot \dot{x}$}}ds = \dot{t}_0+2c_0\!\!\underbrace{\,\int_{s_0}^s \nabla_{\!x}V\cdot \dot{x}\,ds\,}_{\text{$\defeq F(s)$}}.
$$
In particular, if $V_t = 0$, then $F(s) = \int_{s_0}^s\frac{dV}{ds}\,ds = V(x(s)) - V(x_0)$, so that, setting $c_1 \defeq \frac{\dot{t}_0}{2}-c_0V(x_0)$, the $\ddot{x}$-equation becomes
\beqa
\label{eqn:x_end}
\ddot{x} = -c_0\big(2c_0V(x(s))+2c_1\big) \nabla_{\!x}V(x(s)) = -\nabla_{\!x}(c_0V+c_1)^2.
\eeqa

We therefore have the following ``Riemannian Eisenhart Lift," lifting (time-independent) potentials of the form $V^2$ to their ``square roots":

\begin{thm}[Riemannian ``square root" of a time-independent Hamiltonian]
\label{prop:Kahler}
Let $V(x)$ be a $C^2$ function globally defined on $\RR^n$.  For any choice of constants $c_0,c_1$, every solution $x(s)$ of
\beqa
\label{eqn:Riem2}
\ddot{x} = -\nabla_{\!x}(c_0V+c_1)^2
\eeqa
lifts to a geodesic $\gamma(s) = (\gamma^v(s),\gamma^t(s),x(s))$ of the $(n+2)$-dimensional Riemannian metric \eqref{eqn:Riem} defined via $V(x)$ on $\RR^{n+2}$.
\end{thm}

\begin{proof}
Let $x(s)$ have initial data $x(s_0) \defeq x_0, \dot{x}(s_0) \!\defeq \dot{x}_0$.  Given $c_0,c_1$ in \eqref{eqn:Riem2}, choose $\dot{t}_0 = 2c_0V(x_0) +2c_1$ and $\dot{v}_0 = V(x_0)\dot{t}_0 +\frac{1}{2}c_0.$
Then for any choice of $v_0,t_0$, the geodesic $\gamma(s)$ of \eqref{eqn:Riem} starting at $\gamma(s_0)=(v_0,t_0,x_0)$ with initial velocity $\gamma'(s_0) = (\dot{t}_0,\dot{v}_0,\dot{x}_0)$ will, as we saw in \eqref{eqn:x_end}, have $x^i$-components $\gamma^i(s)$ satisfying $\ddot{\gamma}^i = -\partial_i(c_0V+c_1)^2$.  This is precisely \eqref{eqn:Riem2}.
\end{proof}

Regarding the domains of $x(s)$ and $\gamma(s)$: By \cite{gordon,ebin}, it is well known that $\ddot{x} = -\nabla_{\!x}(c_0V+c_1)^2$ always has complete solutions when $V(x)$ is globally defined on $\RR^n$.  Moreover, as shown in \cite[Proposition~1]{aazami}, \eqref{eqn:Riem} is always geodesically complete when $V(t,x)$ is globally defined (even if $V_t \neq 0$).  Let us now show that, for the following class of generalized Hamiltonian systems, the ``Riemannian Eisenhart Lift" of Theorem \ref{prop:Kahler} yields a two-point boundary result and generalizes the well known conservation of energy equation:

\begin{thm}[Generalized Hamiltonian system]
\label{thm:Kahler}
Let $V(t,x)$ be a $C^2$ function globally defined on $\RR\times \RR^n$.  For any two points $x_0,x_1 \in \RR^n$, there exist a constant $c$ and a $C^2$ function $\tau(t)$, both depending on $x_0,x_1$, such that the second-order ODE
\beqa
\label{eqn:tV}
\ddot{x} = -2\Big(c+\!\!\int_0^t(\nabla_{\!x}\widetilde{V}\cdot \dot{x})\,dt\Big)\nabla_{\!x}\widetilde{V} \comma \widetilde{V}(t,x) \defeq V(\tau(t),x),
\eeqa
has a complete solution $x(t)$ passing through $x_0$ and $x_1$, which satisfies
\beqa
\label{eqn:CoE}
\frac{1}{2}|\dot{x}(t)|^2 + \frac{1}{4}\dot{\tau}(t)^2 = \emph{\text{const.}}
\eeqa
If $V_t = 0$, then \eqref{eqn:tV} reduces to the Hamiltonian system
\beqa
\label{eqn:V^2}
\ddot{x} = -\nabla_{\!x}(V+\bar{c})^2 \comma \bar{c}\defeq c-V(x_0),
\eeqa
and \eqref{eqn:CoE} the conservation of energy equation for the potential $(V+\bar{c})^2$.
\end{thm}

\begin{proof}
As mentioned above, for any $C^2$ function $V(t,x)$ globally defined on $\RR\times \RR^n$, the corresponding Riemannian manifold $(\RR^{n+2},\gR)$ defined by \eqref{eqn:Riem} is geodesically complete. By the Hopf-Rinow Theorem, any two points in $\RR^{n+2}$ can thus be connected by a $C^2$-geodesic segment of $\gR$ (in fact a length-minimizing one with respect to the Riemannian distance function $d_{\text{$\gR$}}$ on $\RR^{n+2}$ induced by $\gR$; see, e.g., \cite[Corollary~6.21]{Lee}).  Let us take $(0,0,x_0), (v_1,1,x_1) \in  \RR^{n+2}$, where $v_1$ is any number such that
\beqa
\label{eqn:v_not}
v_1 \neq \int_0^1 V(u,(x_1-x_0)u+x_0)\,du.
\eeqa
Now let $\gamma(s) = (\gamma^v(s),\gamma^t(s),\gamma^1(s),\dots,\gamma^n(s)) \defeq (\gamma^v(s),\gamma^t(s),\gamma^x(s))$ be a geodesic of $(\RR^{n+2},\gR)$ connecting $(0,0,x_0)$ and $(v_1,1,x_1)$; by linearly rescaling the affine parameter $s$ if necessary, we may ensure that $\gamma(0) = (0,0,x_0)$ and $\gamma(1) = (v_1,1,x_1)$ (note that a linear rescaling of $s$ does not affect the status of $\gamma(s)$ as a geodesic of $\gR$).  Recalling \eqref{eqn:Kahler} and \eqref{eqn:v_end}, the geodesic $\gamma\colon \RR\lra \RR^{n+2}$, which has some initial velocity $\dot{\gamma}(0) = (\dot{\gamma}^v_0,\dot{\gamma}^t_0,\dot{\gamma}^1_0,\dots,\dot{\gamma}^n_0)$, satisfies
\beqa
\left.\begin{array}{ccc}
\ddot{\gamma}^v \!\!&=&\!\! \frac{d}{ds}(V\dot{\gamma}^t),\phantom{\Big(\Big)}\\
\ddot{\gamma}^t \!\!&=&\!\! 2c_0\Big(\frac{dV}{ds}-V_t\,\dot{\gamma}^t\Big),\phantom{\Big(\Big)}\\
\ddot{\gamma}^i \!\!&=&\!\! -c_0V_i\,\dot{\gamma}^t,\phantom{\Big(\Big)}\\
\end{array}\right\} \comma c_0 = 2\dot{\gamma}^v_0 - 2V(0,x_0)\dot{\gamma}^t_0.\label{eqn:c=0}
\eeqa
Integrating $\ddot{\gamma}^t$ yields 
$\dot{\gamma}^t = \dot{\gamma}^t_0+2c_0\int_{0}^s (\nabla_{\!x}V\cdot \dot{\gamma}^x)\,ds$, and thus
\beqa
\label{eqn:c_01}
\ddot{\gamma}^x = -c_0\Big( \dot{\gamma}^t_0+2c_0\!\!\int_{0}^s (\nabla_{\!x}V\cdot \dot{\gamma}^x)\,ds\Big) \nabla_{\!x}V(\gamma^t(s),\gamma^x(s)).
\eeqa
We now show that, by a further rescaling of $\gamma(s)$ if necessary, $c_0$ can always be chosen to equal $1$. To begin with, suppose that $\dot{\gamma}^v_0$ and $\dot{\gamma}^t_0$ of $\gamma(s)$ are such that $c_0 = 0$; then $\gamma^t(s) = s$ and
$$
\gamma^x(s) = (x_1-x_0)s+x_0 \commas \dot{\gamma}^v(s) = V(s,(x_1-x_0)s+x_0) +\underbrace{\,\dot{\gamma}^v_0- V(0,x_0)\,}_{\text{$=0$ by \eqref{eqn:c=0}}}.
$$
In particular, $v_1 = \gamma^v(1) = \int_0^1 V((x_1-x_0)u+x_0)du$, which is impossible by \eqref{eqn:v_not}.  Having thus ensured that $c_0\neq 0$, the linear rescaling $\tilde{\gamma}(s) \defeq \gamma(c_0^{-1}s)$ yields
$$
\tilde{\gamma}'(0) = c_0^{-1}\gamma'(0) \imp \tilde{c}_0 \overset{\eqref{eqn:c=0}}{=} 2\dot{\tilde{\gamma}}^v_0 - 2V(0,x_0)\dot{\tilde{\gamma}}^t_0 = 1.
$$
Thus $\tilde{\gamma}(s)$ passes through $\tilde{\gamma}(0) = (0,0,x_0)$ and $\tilde{\gamma}(c_0)=(v_1,1,x_1)$ and has $\tilde{c}_0 = 1$.  Inserting this information into \eqref{eqn:c_01} and changing notation $s \mapsto t$ yields \eqref{eqn:tV} with $c \defeq \dot{\tilde{\gamma}}^t_0/2, \tau(t) \defeq \tilde{\gamma}^t(t)$, and $x(t) \defeq \tilde{\gamma}^x(t)$. In particular, $x(t)$ passes through $x_0$ and $x_1$, and is complete because $\tilde{\gamma}(s)$ is so. We now verify the constant of the motion \eqref{eqn:CoE}. Replacing $\dot{\tilde{\gamma}}^t$ with $\dot{\tau}$ and $\dot{\tilde{\gamma}}^i$ with $\dot{x}^i$ in \eqref{eqn:c=0}, note that
$$
\sum_{i=1}^n \ddot{x}^i\dot{x}^i = -\!\sum_{i=1}^n V_i\,\dot{\tau}\,\dot{x}^i = -\dot{\tau}\Big(\frac{dV}{ds}-V_t\,\dot{\tau}\Big) = -\frac{1}{2}\ddot{\tau}\dot{\tau},
$$
from which \eqref{eqn:CoE} follows. Finally, if $V_t = 0$, then $\dot{\tau} = \dot{\tau}_0+2\big(V(x(s)) - V(x_0)\big)$, hence $\frac{1}{4}\dot{\tau}^2 = (V+\bar{c})^2$ with $\bar{c} \defeq \frac{\dot{\tau}_0}{2} -V(x_0) = c - V(x_0)$. The remainder of the proof now follows.
\end{proof}

By way of comparison, recall that the ``Bolza two-point boundary problem" for a Hamiltonian system asks whether, given two points $x_0,x_1 \in \RR^n$, there exists at least one solution of $\ddot{x} = -\nabla_{\!x}V$ connecting $x_0$ and $x_1$.  Although this is a long-standing question, it was, to the best of our knowledge, solved only recently for time-independent potentials, in \cite[Corollary~4.1]{bartolo}. Specifically, using Jacobi metrics and conservation of energy $E = \frac{1}{2}|\dot{x}(t)|^2+V(x(t))$, \cite[Corollary~4.1]{bartolo} showed, for each $x_0,x_1 \in \RR^n$, the existence of an energy $E_0>0$ such that, for all $E\geq E_0$, there is a solution of $\ddot{x} = -\nabla_{\!x}V$ with energy $E$ from $x_0$ to $x_1$.  (For results prior to this, see, e.g., \cite[Theorem~1]{ekeland} and \cite[Theorem~2]{bolle}.) As for the time-dependent case, \cite[Theorem~1.1]{Bolza_AMJ} showed that for a given $C^1$-potential $V\colon [0,\delta] \times \RR^n \lra \RR$, if there exist $\bar{x} \in \RR^n, p \in [0,2)$ and $\lambda,\mu,k \in \RR$ such that $V(t,x)$ satisfies the at-most-quadratic inequality
$$
V(t,x) \leq \lambda d_0(x,\bar{x})^2 +  \mu d_0(x,\bar{x})^p + k \comma d_0(x,\bar{x}) \defeq |x-\bar{x}|,
$$
for all $x \in \RR^n$, then there is a solution of $\ddot{x} = -\nabla_{\!x}V$ connecting any two distinct points in $\RR^n$ (in fact this result also applies to any complete Riemannian manifold $(M,g)$ with distance function $d_{\scalebox{0.6}{$g$}}$ in place of $d_0$).

\section*{Acknowledgments}
The author thanks Graham Cox for helpful discussions on \cite{forst} and Miguel S\'anchez for helpful discussions on \cite{bartolo}. Finally, the author thanks Matthias Blau for helpful discussions and warmly acknowledges the hospitality of the Albert Einstein Center at Universt\"at Bern.

\bibliographystyle{alpha}
\bibliography{Eisenhart_lift}

\newcommand{\etalchar}[1]{$^{#1}$}
\begin{thebibliography}{GHKW11}

\bibitem[AR22]{aazami}
Amir~Babak Aazami and Robert Ream.
\newblock Almost {K}{\"a}hler metrics and pp-wave spacetimes.
\newblock {\em Letters in Mathematical Physics}, 112(4):84, 2022.

\bibitem[BEE96]{beem}
John~K. Beem, Paul~E. Ehrlich, and Kevin~L. Easley.
\newblock {\em Global {L}orentzian {G}eometry}.
\newblock Marcel Dekker, Inc., $2^{\text{nd}}$ edition, 1996.

\bibitem[Ben97]{benenti}
Sergio Benenti.
\newblock Intrinsic characterization of the variable separation in the
  {H}amilton--{J}acobi equation.
\newblock {\em Journal of Mathematical Physics}, 38(12):6578--6602, 1997.

\bibitem[BGS02]{bartolo}
Rossella Bartolo, Anna Germinario, and Miguel S{\'a}nchez.
\newblock Convexity of domains of {R}iemannian manifolds.
\newblock {\em Annals of Global Analysis and Geometry}, 21:63--84, 2002.

\bibitem[Bol99]{bolle}
Philippe Bolle.
\newblock On the {B}olza problem.
\newblock {\em Journal of Differential Equations}, 152(2):274--288, 1999.

\bibitem[Bri25]{brinkmann}
H.W. Brinkmann.
\newblock Einstein spaces which are mapped conformally on each other.
\newblock {\em Mathematische Annalen}, 94(1):119--145, 1925.

\bibitem[BS18]{BernSuhr}
Patrick Bernard and Stefan Suhr.
\newblock Lyapounov functions of closed cone fields: from {C}onley theory to
  time functions.
\newblock {\em Communications in Mathematical Physics}, 359:467--498, 2018.

\bibitem[CA15]{cariglia}
Marco Cariglia and Filipe~Kelmer Alves.
\newblock The {E}isenhart lift: a didactical introduction of modern geometrical
  concepts from {H}amiltonian dynamics.
\newblock {\em European Journal of Physics}, 36(2):025018, 2015.

\bibitem[Car14]{cariglia2}
Marco Cariglia.
\newblock Hidden symmetries of dynamics in classical and quantum physics.
\newblock {\em Reviews of Modern Physics}, 86(4):1283, 2014.

\bibitem[CFS03a]{candela}
Anna~Maria Candela, Jos{\'e}~L. Flores, and Miguel S{\'a}nchez.
\newblock On general plane fronted waves. {G}eodesics.
\newblock {\em General Relativity and Gravitation}, 35(4):631--649, 2003.

\bibitem[CFS03b]{Bolza_AMJ}
Anna~Maria Candela, Jos{\'e}~L. Flores, and Miguel S{\'a}nchez.
\newblock A quadratic {B}olza-type problem in a {R}iemannian manifold.
\newblock {\em Journal of Differential Equations}, 193(1):196--211, 2003.

\bibitem[CPC00]{casetti2}
Lapo Casetti, Marco Pettini, and E.~G.~D. Cohen.
\newblock Geometric approach to {H}amiltonian dynamics and statistical
  mechanics.
\newblock {\em Physics Reports}, 337(3):237--341, 2000.

\bibitem[CS08]{Sanchez}
Anna~Maria Candela and Miguel S{\'a}nchez.
\newblock Geodesics in semi-{R}iemannian manifolds: geometric properties and
  variational tools.
\newblock {\em Recent developments in pseudo-Riemannian geometry}, 4:359, 2008.

\bibitem[DBKP85]{barg}
C.~Duval, G.~Burdet, H.~P. K{\"u}nzle, and M.~Perrin.
\newblock Bargmann structures and {N}ewton-{C}artan theory.
\newblock {\em Physical Review D}, 31(8):1841, 1985.

\bibitem[DGH91]{duval_gibbons}
Christian Duval, Gary Gibbons, and P{\'e}ter Horv{\'a}thy.
\newblock Celestial mechanics, conformal structures, and gravitational waves.
\newblock {\em Physical Review D}, 43(12):3907, 1991.

\bibitem[Ebi70]{ebin}
David~G. Ebin.
\newblock Completeness of {H}amiltonian vector fields.
\newblock {\em Proceedings of the American Mathematical Society},
  26(4):632--634, 1970.

\bibitem[EGT96]{ekeland}
I.~Ekeland, N.~Ghoussoub, and H.~Tehrani.
\newblock Multiple solutions for a classical problem in the calculus of
  variations.
\newblock {\em Journal of Differential Equations}, 131(2):229--243, 1996.

\bibitem[Eis28]{eisenhart}
Luther~Pfahler Eisenhart.
\newblock Dynamical trajectories and geodesics.
\newblock {\em Annals of Mathematics}, pages 591--606, 1928.

\bibitem[For96]{forst}
Franc Forstneric.
\newblock Actions of $(\mathbb{R},+)$ and $(\mathbb{C},+)$ on {C}omplex
  {M}anifolds.
\newblock {\em Mathematische Zeitschrift}, 223(1):123--153, 1996.

\bibitem[FS03]{flores2}
Jos{\'e}~L. Flores and Miguel S{\'a}nchez.
\newblock Causality and conjugate points in general plane waves.
\newblock {\em Classical and Quantum Gravity}, 20(11):2275, 2003.

\bibitem[FS12]{fathi}
Albert Fathi and Antonio Siconolfi.
\newblock On smooth time functions.
\newblock In {\em Mathematical Proceedings of the Cambridge Philosophical
  Society}, volume 152, pages 303--339. Cambridge University Press, 2012.

\bibitem[GHKW11]{gibbons_s}
G.~W. Gibbons, T.~Houri, D.~Kubiz{\v{n}}{\'a}k, and C.~M. Warnick.
\newblock Some spacetimes with higher rank {K}illing--{S}t{\"a}ckel tensors.
\newblock {\em Physics Letters B}, 700(1):68--74, 2011.

\bibitem[Gib20]{gibbons2}
G.~W. Gibbons.
\newblock Lifting the {E}isenhart-{D}uval lift to a minimal brane.
\newblock {\em arXiv:2003.06179}, 2020.

\bibitem[GL16]{globke}
Wolfgang Globke and Thomas Leistner.
\newblock Locally homogeneous pp-waves.
\newblock {\em Journal of Geometry and Physics}, 108:83--101, 2016.

\bibitem[Gor70]{gordon}
William~B. Gordon.
\newblock On the completeness of {H}amiltonian vector fields.
\newblock {\em Proceedings of the American Mathematical Society},
  26(2):329--331, 1970.

\bibitem[Gor74]{gordon2}
William~B. Gordon.
\newblock The existence of geodesics joining two given points.
\newblock {\em Journal of Differential Geometry}, 9(3):443--450, 1974.

\bibitem[Lee18]{Lee}
John~M. Lee.
\newblock {\em Introduction to {R}iemannian manifolds}.
\newblock Springer, $2{\text{nd}}$ edition edition, 2018.

\bibitem[LS16]{Leistner}
Thomas Leistner and Daniel Schliebner.
\newblock Completeness of compact {L}orentzian manifolds with abelian holonomy.
\newblock {\em Mathematische Annalen}, 364(3-4):1469--1503, 2016.

\bibitem[LS23]{lejmi}
Mehdi Lejmi and Xi~Sisi Shen.
\newblock Canonical almost-{K}{\"a}hler metrics dual to general plane-fronted
  wave {L}orentzian metrics.
\newblock {\em Mathematische Zeitschrift}, 303(4):94, 2023.

\bibitem[Min07]{minguzzi2}
Ettore Minguzzi.
\newblock Eisenhart's theorem and the causal simplicity of {E}isenhart's
  spacetime.
\newblock {\em Classical and Quantum Gravity}, 24(11):2781, 2007.

\bibitem[Min12]{minguzzi3}
Ettore Minguzzi.
\newblock Causality of spacetimes admitting a parallel null vector and weak
  {KAM} theory.
\newblock {\em arXiv:1211.2685}, 2012.

\bibitem[Mon16]{monclair}
Daniel Monclair.
\newblock Attractors in spacetimes and time functions.
\newblock {\em arXiv:1603.06994}, 2016.

\bibitem[MS08]{minguzzi}
Ettore Minguzzi and Miguel S{\'a}nchez.
\newblock The causal hierarchy of spacetimes.
\newblock {\em Recent developments in pseudo-Riemannian geometry}, 4:299--358,
  2008.

\bibitem[O'N83]{o1983}
Barrett O'Neill.
\newblock {\em Semi-{R}iemannian {G}eometry with {A}pplications to
  {R}elativity}.
\newblock Academic press, 1983.

\bibitem[Pet07]{pettini}
Marco Pettini.
\newblock {\em Geometry and {T}opology in {H}amiltonian {D}ynamics and
  {S}tatistical {M}echanics}.
\newblock Springer, 2007.

\bibitem[Pin75]{ong}
Ong~Chong Pin.
\newblock Curvature and {M}echanics.
\newblock {\em Advances in Mathematics}, 15(3):269--311, 1975.

\bibitem[SHN{\etalchar{+}}17]{AMS}
Christina Sormani, Denson~C. Hill, Pave{\l} Nurowski, Lydia Bieri, David
  Garfinkle, and Nicol{\'a}s Yunes.
\newblock The {M}athematics of {G}ravitational {W}aves: {A} {T}wo-{P}art
  {F}eature.
\newblock {\em Notices of the AMS}, 64(7):684--707, 2017.

\bibitem[Sul76]{sullivan}
Dennis Sullivan.
\newblock Cycles for the dynamical study of foliated manifolds and complex
  manifolds.
\newblock {\em Inventiones Mathematicae}, 36(1):225--255, 1976.

\bibitem[Tab84]{taber}
William Taber.
\newblock Projectively equivalent metrics subject to constraints.
\newblock {\em Transactions of the American Mathematical Society},
  282(2):711--737, 1984.

\end{thebibliography}
\end{document}